\documentclass[11p,reqno]{amsart}
\usepackage{amsmath}
\usepackage{amsfonts}
\usepackage{amssymb}
\usepackage{graphicx}
\usepackage{color,soul}
\usepackage{enumitem}
\newtheorem{theorem}{Theorem}[section]\newtheorem{thm}[theorem]{Theorem}
\newtheorem*{theorem*}{Theorem}
\newtheorem{lemma}{Lemma}[section]
\newtheorem{corollary}[theorem]{Corollary}
\newtheorem{proposition}{Proposition}[section]

\newtheorem{definition}[theorem]{Definition}

\newtheorem{remark}[theorem]{Remark}

\def\l{\lambda}

\def\p{\partial}

\def\G{\Gamma}

\def\L{{\mathcal L}}

\numberwithin{equation}{section}

\begin{document}

\title[]{On the Lower Bound of the Principal Eigenvalue of a Nonlinear Operator}


\author{Yucheng Tu}
\address{Department of Mathematics, University of California, San Diego, La Jolla, CA 92093, USA}
\email{y7tu@ucsd.edu}




\maketitle

\begin{abstract}   
We prove sharp lower bound estimates for the first nonzero eigenvalue of the non-linear elliptic diffusion operator $L_p$ on a smooth metric measure space, without boundary or with a convex boundary and Neumann boundary condition, satisfying $BE(\kappa,N)$ for $\kappa\neq 0$. 
Our results extends the work of Koerber\cite{Valtorta12} for case $\kappa=0$ and Naber-Valtorta\cite{NV14} for the $p$-Laplacian.
\end{abstract}

\section{Introduction}

Let $M$ be a compact manifold. The Laplacian operator on $M$ plays a key role in studying the geometry of $M$, and one of the key quantity related to the Laplacian is its first nonzero eigenvalue $\lambda_1$, also called the principal eigenvalue. There have been a lot of works on the estimate of $\lambda_1$ for the Neumann boundary value problem:
$$\begin{cases}
\Delta u=-\lambda u\qquad\text{ on }M\\
\frac{\partial u}{\partial \nu}=0\qquad\text{ in }\partial M
\end{cases}$$
In \cite{PW60} Payne and Weinberger showed $\lambda_1\geq \pi^2/D^2$ for the Laplacian on the convex subset of $\mathbb{R}^n$ with diameter $D$. Later in \cite{Ch70} Cheeger gave a lower bound of $\lambda_1$ in terms of the isoperimetric constant on compact Riemannian manifolds. In \cite{LY80}, given that $M$ has nonnegative Ricci curvature, P. Li and Yau proved the lower bound $\pi^2/4D^2$ by using gradient estimate. Later Zhong and Yang \cite{ZY84} used a barrier argument to prove the sharp lower bound $\pi^2/D^2$ for compact Riemannian manifold with nonnegative Ricci curvature. Afterwards Kroger in \cite{Kr92} used a gradient comparison technique to recover the result of Zhong and Yang, and furthermore he was able to deal with a negative Ricci lower bound case. It was in Bakry and Emery's work \cite{BE86} that the situation is generalized into a manifold with weighted volume measure, and the Laplacian is replaced by a general elliptic diffusion operator $L$. By defining curvature-dimension condition, we can make sense of Ricci lower bound in the senario of smooth measure spaces. Later Bakry and Qian \cite{BQ00} used gradient comparison technique similar to Kroger to prove the sharp lower bound $\pi^2/D^2$ for $\lambda_1(L)$ assuming $M$ to be $BE(0,N)$ for some $N\geq 1$. Later Andrews and Ni \cite{AN12} recovered this result with a simple modulus of continuity method. 

In recent years there is much attention to the nonlinear operator called $p-$Laplacian $\Delta_p$. In \cite{Valtorta12} he showed the sharp estimate $\lambda_1\geq (p-1)\frac{\pi_p^p}{D^p}$ for Riemannian manifolds with Ricci lower bound $0$ and $p>1$, where $\pi_p$ is the half period of $p-$sine function which will be defined later. The method used is a gradient comparison via Bochner formula for $p$-Laplacian, and a fine ODE analysis of the one dimensional model solution. Later Naber and Valtorta \cite{NV14} extended the result to the case $\text{Ric}\geq \kappa(n-1)$ for $\kappa<0$. The key improvement is the better understanding of the one dimensional model equation in the $\kappa<0$ case, which is considerably more complicated than non-negative case. Very recently Li-Wang in \cite{LW19eigenvalue} and \cite{LW19eigenvalue2} used a modulus of continuity method to get the gradient comparison, in the case of drifted $p-$Laplacian, which fits in the setting of a Bakry-Emery manifold with weight $e^{-f}$, thus opening the possibility of studying the non-linear version of $L$ operator with drifted terms in metric measure spaces Satisfying $BE(\kappa,N)$. For $\kappa=0$ case, \cite{Ko18} showed that $(p-1)\frac{\pi_p^p}{D^p}$ is the sharp lower bound. 

In this paper we follow the approaches of \cite{BQ00} and \cite{NV14} to study the non-linear operator $L_p$ on a compact manifold with possibly convex boundary(to be defined later). We extend the Theorem 1.1 of \cite{Ko18} to the $\kappa\neq 0$ case, more precisely:

\begin{theorem}
\label{Main Theorem}
Let $M$ be compact and connected and $L$ be an elliptic diffusion operator with invariant measure $m$. Assume that $L$ satisfies $BE(\kappa,N)$ where $\kappa\neq 0$. Let $D$ be diameter defined by the intrinsic distance metric on $M$. Let $u$ be an eigenfunction associated with $\lambda$
satisfying Neumann boundary condition if $\p M \neq \emptyset$, where $\lambda$ is the first nonzero eigenvalue of $L_p$. Then denoting $w^{(p-1)} := |w|^{p-2}w$, we have
\begin{itemize}[leftmargin=0.5in]
    \item[(1)] When $\kappa>0$, assuming further that $D\leq \pi/\sqrt{\kappa}$, we have a sharp comparison:
    $$\lambda\geq \lambda_D$$
    where $\lambda_D$ is the first nonzero eigenvalue of the following Neumann eigenvalue problem on $[-D/2,D/2]$:
    $$\frac{d}{dt}\big[(w')^{(p-1)}\big]-(n-1)\sqrt{\kappa}\tan(\sqrt{\kappa} t)(w')^{(p-1)}+\lambda w^{(p-1)}=0$$
    \item[(2)] When $\kappa<0$, we have a sharp comparison:
    $$\lambda\geq \lambda_D$$
    where $\lambda_D$ is the first nonzero eigenvalue of the following Neumann eigenvalue problem on $[-D/2,D/2]$:
    $$\frac{d}{dt}\big[(w')^{(p-1)}\big]+(n-1)\sqrt{-\kappa}\tanh(\sqrt{-\kappa t})(w')^{(p-1)}+\lambda w^{(p-1)}=0$$
\end{itemize}
\end{theorem}

The rest of this paper is devoted to the proof of Theorem \ref{Main Theorem}. The basic structure is the following. In section 2 we introduce the setting and definitions related to the linear elliptic diffusion operator $L$. In section 3 we define the non-linear operator $L_p$ and its Neumann eigenvalue problem. In section 4 we use the Bochner formula to derive a useful estimate (Prop 4.2) which will be used in proving the gradient comparison theorem in section 5. In section 6 we study the associated three one-dimensional model equations and combined with section 7, we get maximum comparison between our model solutions and the solution to the Neumann eigenvalue problem. Finally combining the gradient, maximum and diameter comparison we prove the theorem in section 8.
\vspace{0.5cm}\\
\textbf{Acknowledgment.} The author would like to thank his advisor Professor Lei Ni for lots of encouragement and helpful suggestions, and Dr. Xiaolong Li for explaining his paper with Kui Wang \cite{LW19eigenvalue} and \cite{LW19eigenvalue2} to him.
\section{The geometry of elliptic diffusion operators}
In this section we give some definitions which will be used later in our proof. First we introduce the elliptic diffusion operator, which is a natural generalization of a second order linear differential operator on a Riemannian manifold.

\begin{definition}
A linear second order operator $L:C^{\infty}(M)\to C^{\infty}(M)$ is called an elliptic diffusion operator if for any $\Phi:\mathbb{R}^r\to\mathbb{R}$ we have 
$$L(\Phi(f_1,f_2,\dots,f_r))=\sum_{i=1}^r\partial_i\Phi L(f_i)+\sum_{i,j=1}^r\partial_i\partial_j\Phi\Gamma(f_i,f_j)$$
and $\Gamma(f,f)\geq 0$ with equality if and only if $df=0$. Here $\Gamma$ is defined as
$$\Gamma(f,g):=\frac{1}{2}\big(L(fg)-fLg-gLf\big).$$
\end{definition}

\begin{definition}
We say that a locally finite Borel measure $m$ is L-invariant if there is a generalized function $\nu$ such that
$$\int_M\Gamma(f,h)dm=-\int_MfLhdm+\int_{\partial M}f\Gamma(g,\nu)dm$$
holds for all smooth $f,g$. $\nu$ is called the outward normal function and is defined to be a set of pairs $(\nu_i, U_i)_{i\in I}$ for a covering $U_i$ of $\partial M$ such that $\nu_i\in C^\infty(U_i)$ and $\Gamma(\nu_i-\nu_j, \cdot)|_{U_i\cap U_j}=0$.
\end{definition}

\begin{definition}
We define the intrinsic distance $d:M\times M\to[0,\infty]$ as:
$$d(x,y):=\sup\Big\{f(x)-f(y)|f\in C^\infty(M),\Gamma(f)\leq 1\Big\}$$
and the diameter of $M$ by $D:=\sup\{d(x,y)|x,y\in M\}$.
\end{definition}

\begin{definition}
For any $f$, $u,v\in C^\infty(M)$, we define the Hessian by
$$H_f(u,v)=\frac{1}{2}\Big(\Gamma(u,\Gamma(f,v))+\Gamma(v,\Gamma(f,u))-\Gamma(f,\Gamma(u,v))\Big)$$
and the $\Gamma_2$-operator by
$$\Gamma_2(u,v)=\frac{1}{2}\Big(L(\Gamma(u,v))-\Gamma(u,Lv)-\Gamma(v,Lu)\Big).$$
\end{definition}

\begin{definition}
We can define the $N$-Ricci curvature as
$$\text{Ric}_N(f,f)(x)=\inf\Big\{\Gamma_2(\phi,\phi)(x)-\frac{1}{N}(L\phi)^2(x)\Big|\phi\in C^\infty(M),\Gamma(\phi-f)(x)=0\Big\}$$
and let $\text{Ric}=\text{Ric}_\infty$.
\end{definition}
Let $\kappa\in\mathbb{R}$ and $N\in [1,\infty]$, we say that $L$ satisfies $BE(\kappa,N)$ condition if and only if
$$\text{Ric}_N(f,f)\geq \kappa\Gamma(f).$$
for any $f\in C^\infty(M)$.
\vspace{0.2cm}

If $M$ has a boundary, the geometry of $\partial M$ also plays an important role in the eigenvalue estimate. We define the convexity of $\partial M$ as follows.

\begin{definition}
Let $\nu$ be the outward normal direction, and $U\subset M$ be an open set, $\phi$,$\eta\in C^\infty(U)$ such that $\Gamma(\nu,\eta)$ and $\Gamma(\nu,\phi)=0$ on $U\cap\partial M$. We define the second fundamental form on $\partial M$ by 
$$II(\phi,\eta)=-H_\phi(\eta,\nu)=-\frac{1}{2}\Gamma(\nu,\Gamma(\eta,\phi)).$$
If for any $\phi$ as above with $\Gamma(\phi)>0$ on $U\cap \partial M$ we have
$$II(\phi,\phi)\leq 0\qquad \text{ on }U\cap\partial M$$
Then we say $\partial M$ is convex. If we have strict inequality then $\partial M$ is called strictly convex.
\end{definition}

\section{The generalized $p$-Laplacian and its eigenvalue problem} 

Now we are going to work on the eigenvalue problem of the non-linear operator $L_p$ derived from the previously defined $L$. The generalized $p$-Laplacian is defined by 

\begin{equation*}
 L_p u(x)=\begin{cases} \G(u)^{\frac{p-2}{2}} \left(Lu +(p-2)\frac{H_u(u,u)}{\G(u)} \right) & \text{ if } \G(u)(x) \neq 0 ;\\
 0 & \text{ otherwise. }
 \end{cases}   
\end{equation*}

We also define 
\begin{equation*}
    \L_p^u (\eta) =\begin{cases} \G(u)^{\frac{p-2}{2}} \left(L\eta +(p-2)\frac{H_{\eta}(u,u)}{\G(u)} \right) & \text{ if } \G(u)(x) \neq 0 ;\\
 0 & \text{ otherwise. }
 \end{cases}   
\end{equation*}
which is the linearization of $L_p$. Now we define the eigenvalue of $L_p$. If $\lambda\in\mathbb{R}$ and $u\in C^2(M)$ satisfies the Neumann boundary problem:
$$\begin{cases}
L_pu=-\lambda u|u|^{p-2}\qquad\text{ on } M^\circ\\
\Gamma(u,\tilde{\nu})=0\qquad\text{ on }\partial M
\end{cases}$$
Then we call $\lambda$ an eignevalue, and $u$ an eigenfunction of $L_p$, however, we may not always find a classical solution. To define the eigenfunction in a weak sense, we first use the invariance of $m$ to deduce the following integration-by-parts formula:

\begin{lemma}
Let $\phi\in C^\infty(M)$ and $u\in C^2(M)$ and $\Gamma(u)>0$ on $\text{supp}(\phi)$. Then we have
$$\int_M\phi L_pudm=-\int_M\Gamma(f)^\frac{p-2}{2}\Gamma(f,\phi)dm+\int_{\partial M}\Gamma(f,\tilde{\nu})\Gamma(f)^\frac{p-2}{2}\phi dm$$
\end{lemma}

So we define the eigenvalue and eigenfunction by
\begin{definition}
We say that $\lambda$ is an eigenvalue of $L_p$ if there is a $u\in W^{1,p}(M)$ such that for any $\phi\in C^\infty(M)$ the following identity holds:
$$\int_M\Gamma(u)^\frac{p-2}{2}\Gamma(u,\phi)dm=\lambda\int_M\phi u|u|^{p-2}dm$$
\end{definition}

We have the following result concerning the regularity of principal eigenfunctions.

\begin{lemma}(Lemma 2.2 in \cite{Ko18})
If $M$ is a compact smooth Riemannian manifold with an elliptic diffusion operator $L$ and an $L$-invariant measure $m$. Then the principal eigenfunction is in $C^{1,\alpha}(M)$ for some $\alpha>0$, and $u$ is smooth near points $x\in M$ such that $\Gamma(u)(x)\neq0$ and $u(x)\neq 0$; for $p<2$, $u$ is $C^{3,\alpha}$,  and for $p>2$, $u$ is $C^{2,\alpha}$ near $x$ where $\Gamma(u)(x)\neq0$ and $u(x)=0$.
\end{lemma}

\section{Bochner formula}
In this section, we will derive the Bochner formula and an estimate which is helpful to prove the gradient estimate in the next section. 

\begin{proposition}[Bochner formula]
Let $u\in C^{1,\alpha}(M)$ be a first eigenfunction of $L_p$, and $x\in M$ be a point such that $\G(u)(x)\neq $ and $u(x)\neq 0$. Then at $x$ we have the following formula: 
\begin{align*}
\frac{1}{p}\L_p^u \left(\G(u)^{\frac p 2} \right) = \G(u)^{\frac{p-2}{2}} \left(\G(L_p u, u) -(p-2)L_p u A_u \right)+\G(u)^{p-2} \left(\G_2(u,u) +p(p-2)A_u^2 \right)
\end{align*}
\end{proposition}

\begin{proof}
c.f.\cite{Ko18}, Lemma 3.1.
\end{proof}

\begin{proposition}
Suppose $L$ satisfies $BE(\kappa, N)$ for some $\kappa\in \mathbb{R}$ and $N\in [1, \infty]$. Then for any $n \geq N$, we have for $n\in (1,\infty)$,
\begin{eqnarray*}
\G(u)^{p-2} \left(\G_2(u,u) +p(p-2)A_u^2\right) \geq \frac{(L_p u)^2}{n} +\frac{n}{n-1} 
\left( \frac{L_p u}{n} -(p-1)\G(u)^{\frac{p-2}{2}} A_u \right)^2 +\kappa \G(u)^{p-1}
\end{eqnarray*}
for $n=\infty$, 
\begin{equation*}
    \G(u)^{\frac p 2}  \left(\G_2(u,u) +p(p-2)A_u^2\right) \geq (p-1)^2\G(u)^{p-2}A_u^2 +\kappa \G(u)^{p-1},
\end{equation*}
for $n=1$, 
\begin{equation*}
    \G(u)^{\frac p 2}  \left(\G_2(u,u) +p(p-2)A_u^2\right) \geq (L_p u)^2 +\kappa \G(u)^{p-1}
\end{equation*}

\end{proposition}

\begin{proof}
Following \cite{Ko18} Lemma 3.3, we can scale $u$ on both sides so that $\Gamma(u)(x)=1$. We can assume $n=N$ since $B(\kappa,N)$ implies $B(\kappa,n)$ for $n\geq N$. When $n=1$, by the curvature-dimension inequality and $Lu=\text{tr}H_u=A_u$, we get
$$\Gamma_2(u,u)+p(p-2)A_u^2\geq \kappa+(Lu)^2+p(p-2)A_u^2=\kappa+(p-1)^2A_u^2=(L_pu)^2+\kappa.$$
When $n=\infty$, we have $\Gamma_2(u,u)\geq\kappa+A_u^2$, therefore $\Gamma_2(u,u)+p(p-2)A_u^2\geq \kappa+(p-1)^2A_u^2$. Now if $1<n<\infty$, for any $v\in C^{\infty}(M)$, by the curvature-dimension inequality we have 
$$
    \Gamma_2(v,v)\geq \kappa\Gamma(v)+\frac{1}{N}(Lv)^2
$$
Now we consider a quadratic form $B(v,v)=\Gamma_2(v,v)-\kappa\Gamma(v)-\frac{1}{N}(Lv)^2$, which is non-negative for any $v\in C^{\infty}(M)$. Let $v=\phi(u)$ where $\phi\in C^{\infty}(\mathbb
{R})$. Then by standard computations, together with the assumption $\Gamma(u)=1$, we have
\begin{align*}
\Gamma(\phi(u))&=(\phi')^2,\qquad L(\phi(u))=\phi'Lu+\phi'',\\
\Gamma_2(\phi(u),\phi(u))&=(\phi')^2\Gamma_2(u,u)+2\phi'\phi''A_u+(\phi'')^2.
\end{align*}
Then we get
\begin{align*}
    B(\phi(u),\phi(u))&=\Gamma_2(\phi(u),\phi(u))-\kappa\Gamma(\phi(u))-\frac{1}{N}(L(\phi(u)))^2\\
    &=(\phi')^2\Gamma_2(u,u)+2\phi'\phi''A_u+(\phi'')^2-\kappa(\phi')^2-\frac{1}{N}\Big[\phi'Lu+\phi''\Big]^2\\
    &=(\phi')^2B(u,u)+2\phi'\phi''(A_u-\frac{Lu}{N})+\frac{N-1}{N}(\phi'')^2
\end{align*}
Since $B(\phi(u),\phi(u))\geq 0$ for any $\phi$, we have non-positive discriminant 
$$B(u,u)\frac{N-1}{N}-\Big(A_u-\frac{Lu}{N}\Big)^2\leq 0$$
Therefore we have
\begin{align*}
    &\Gamma_2(u,u)+p(p-2)A_u^2\\
    =&\kappa+\frac{1}{N}(Lu)^2+B(u,u)+p(p-2)A_u^2\\
    \geq& \kappa+\frac{1}{N}(L_p(u)+(p-2)A_u)^2+\frac{N}{N-1}\Big(A_u-\frac{L_p(u)+(p-2)A_u}{N}\Big)^2+p(p-2)A_u^2\\
    =&\kappa+\frac{1}{N}(L_p(u))^2+\frac{N}{N-1}\Big(\frac{L_p(u)}{N}-(p-1)A_u\Big)^2
\end{align*}
\end{proof}

\section{Gradient Comparison Theorem and Its Applications}

In this section we prove the gradient comparison theorem of the eigenfunction with the solution to the one-dimensional model. 

\begin{thm}
\label{Gradient Comparison}
Let $u$ be a weak solution of 
\begin{equation*}
    L_p u=-\l u^{(p-1)}
\end{equation*}
satisfying Neumann boundary condition if $\p M \neq \emptyset$, where $\lambda$ is the first nonzero eigenvalue of $L_p$. Assume that $L$ satisfies $BE(\kappa,N)$. Let $T_\kappa:I\to\mathbb{R}$ be a function that satisfies $T_\kappa' = T^2/(N-1) + (N-1)\kappa$, and
$w:[a,b]\to \mathbb{R}$ be a solution of the following ODE:
\begin{align}\begin{cases}
\frac{d}{dt}\big[(w')^{(p-1)}\big]-T_\kappa (w')^{(p-1)}+\lambda w^{(p-1)}=0\\
w(a)=-1,\quad w'(a)=0
\end{cases}
\end{align}
such that $w$ is strictly increasing on $[a,b]$ and the range of $u$ is contained the range of $w$. 
Then for all $x\in M$, 
\begin{equation*}
    \G(w^{-1}(u(x))) \leq 1.
\end{equation*}

\end{thm}

\begin{proof}
By scaling $u$ so that $\min(u)=-1$, we can assume that the range of $u$ is contained in the range of $w$. By chain rule of $\Gamma$ what we need to show, equivalently, is
$$\Gamma(u)^{\frac12}(x)\leq w'(w^{-1}(u(x)))$$
for all $x\in M$. Since $T_\kappa$ depends smoothly on $\kappa$, we will first prove that for any $\Tilde{\kappa}<\kappa$, the gradient comparison holds when $T_{\kappa}$ is replaced by $T_{\Tilde{\kappa}}$ in (5.1), and then we can take $\Tilde{\kappa}\to \kappa$. This will give us a room to use proof by contradiction.\\

Now for $c>0$ we denote $\phi_c=(cw'\circ w^{-1})^p$, and consider the function $Z_c:M\to \mathbb{R}$
$$Z_c(x)=\Gamma(u)^{\frac{p}{2}}(x)-\phi_c(u(x))$$
Assume for contradiction that $Z_1(x) > 0$ for some $x\in M$. Let
$$c_0=\inf\{c:Z_c(x)> 0 \text{ for some }x\in M\}$$
By our definition of $c_0$, there is a $x_0\in M$ such that $Z_{c_0}(x_0)=0$ is the maximum of $Z_{c_0}$. Now we denote $Z_{c_0}$ as $Z$, $\phi_{c_0}$ as $\phi$ when there is no confusion. When $x_0$ is in the interior of $M$, this clearly implies the following equations:
\begin{align}
   & Z(x_0)=0\\
   & \Gamma(Z,u)(x_0)=0\\
   & \frac{1}{p}\L_p^u(Z)(x_0)\leq 0 
\end{align}

If $x_0\in\partial M$, since $\Gamma(u,\Tilde{\nu})=0$ by the Neumann boundary condition, we have that $\Gamma(Z,u)=0$ at $x_0$. Since $Z$ achieves maximum at $x_0$ and $\partial M$ is convex, we have
\begin{align*}
   0\leq \Gamma(Z,\Tilde{\nu})&=\Gamma(\Gamma(u)^\frac{p}{2}-\phi(u),\Tilde{\nu})=\frac{p}{2}\Gamma(u)^\frac{p-2}{2}\Gamma(\Gamma(u),\Tilde{\nu})-\phi'(u)\Gamma(u,\Tilde{\nu})\\
   &=-p\Gamma(u)^\frac{p-2}{2}II(u,u)-\phi'(u)\cdot 0\leq 0
\end{align*}
Therefore $\Gamma(Z,\Tilde{\nu})(x_0)=0$. This implies that the second derivative of $Z$ along the normal direction is nonpositive. On the other hand, the second derivatives along tangential directions are nonpositive, hence the ellipticity of $\L_p^u$ implies that $\L_p^u(Z)(x_0)\leq 0$. Hence we comfirmed the three equations above for all $x\in M$.\\

From (5.2) we get
$$\frac{p}{2}\Gamma(u)^{\frac{p-2}{2}}\Gamma(\Gamma(u),u)-\phi'(u)\Gamma(u)=0$$
which implies $\phi'(u)=p\Gamma(u)^{\frac{p-2}{2}}A_u$. Now by calculation we have
$$\frac{1}{p}\L_p^u(\phi(u))=\frac{1}{p}\Big(\phi'(u)L_pu+(p-1)\phi''(u)\Gamma(u)^\frac{p}{2}\Big)$$
By chain rule we have $\phi'=p\cdot\big[(w')^{p-2}\cdot w''\big]\circ w^{-1}$, and $\phi''=p\Big[(p-2)(w'')^2+w'''w'\Big]\cdot(w')^{p-4}\circ w^{-1}$, and by differentiating the ODE satisfied by $w$ we have
$$(p-1)(w')^{p-3}\Big[(p-2)(w'')^2+w'''w'\Big]=T_{\Tilde{\kappa}}'(w')^{p-1}+(p-1)T_{\Tilde{\kappa}}w''(w')^{p-2}-\lambda(p-1)w'w^{p-2}$$
Therefore
$$\phi''=p\cdot\frac{T_{\Tilde{\kappa}}'(w')^{p-1}+(p-1)T_{\Tilde{\kappa}}w''(w')^{p-2}-\lambda(p-1)w'w^{p-2}}{w'}\circ w^{-1}.$$
Now we evaluate the above expression at $u(x_0)$. Since $\phi'(u)=p\cdot\big[(w')^{p-2}\cdot w''\big]\circ w^{-1}(u)=p\Gamma(u)^\frac{p-2}{2}A_u$, and by (1) we have $\phi(u)=w'\circ w^{-1}(u)=\Gamma(u)^{\frac{p}{2}}$, we have
\begin{align}\frac{1}{p}\L_p^u(\phi(u))=-\lambda u^{(p-1)}\Gamma(u)^\frac{p-2}{2}A_u+T_{\Tilde{\kappa}}'\Gamma(u)^{p-1}+(p-1)T_{\Tilde{\kappa}}\Gamma(u)^\frac{2p-3}{2}A_u-\lambda(p-1)u^{p-2}\Gamma(u)^\frac{p}{2}
\end{align}
By the ODE evaluated at $w^{-1}(u(x_0))$, we have
$$(p-1)\Gamma(u)^\frac{p-2}{2}A_u-T_{\Tilde{\kappa}}\Gamma(u)^\frac{p-1}{2}+\lambda u^{(p-1)}=0$$
Hence
$$(p-1)T_{\Tilde{\kappa}}\Gamma(u)^\frac{2p-3}{2}A_u=(p-1)\big[(p-1)\Gamma(u)^\frac{p-2}{2}A_u+\lambda u^{(p-1)}\big]\Gamma(u)^\frac{p-2}{2}A_u.$$
Plugging the above equation into the third term of $(5.5)$, we have
\begin{align*}
\frac{1}{p}\L_p^u(\phi(u))=\lambda(p-2)u^{(p-1)}\G(u)^{\frac{p-2}{2}}A_u+T_{\Tilde{\kappa}}'\Gamma(u)^{p-1}+(p-1)^2\Gamma(u)^{p-2}A_u^2-\lambda(p-1)u^{p-2}\Gamma(u)^\frac{p}{2}
\end{align*}

We can assume that $\Tilde{\kappa}$ has the same sign as $\kappa$ since we can pick $\Tilde{\kappa}$ sufficiently close to $\kappa$. Now we have
$T_{\Tilde{\kappa}}'=T_{\Tilde{\kappa}}^2/(n-1)+{\Tilde{\kappa}}$ to rewrite the second term and finally get
\begin{align*}
    \frac{1}{p}\L_p^u(\phi(u))=&-\lambda u^{(p-1)}\Gamma(u)^\frac{p-2}{2}A_u+\frac{1}{n-1}\Big[\lambda u^{(p-1)}+(p-1)\Gamma(u)^\frac{p-2}{2}A_u\Big]^2+{\Tilde{\kappa}}\Gamma(u)^{p-1}\\
    &\quad+(p-1)^2\Gamma(u)^{p-2}A_u^2+(p-1)\lambda u^{(p-1)}\Gamma(u)^\frac{p-2}{2}A_u-\lambda(p-1)u^{p-2}\Gamma(u)^\frac{p}{2}\\
    =& (p-2)\lambda u^{(p-1)}\Gamma(u)^\frac{p-2}{2}A_u-\lambda(p-1)u^{p-2}\Gamma(u)^\frac{p}{2}+{\Tilde{\kappa}}\Gamma(u)^{p-1}\\
    &\quad +\frac{n}{n-1}\Big[\frac{\lambda u^{(p-1)}}{n}+(p-1)\Gamma(u)^\frac{p-2}{2}A_u\Big]^2+\frac{\lambda^2u^{2p-2}}{n}
\end{align*}
By Proposition 4.1 and 4.2, together with the fact that $L_pu = -\lambda u^{(p-1)}$ we have
\begin{align*}
    \frac{1}{p}\L_p^u \big(\G(u)^{\frac p 2}\big) & = \G(u)^{\frac{p-2}{2}} (\G(L_p u, u) -(p-2)L_p u A_u)+\G(u)^{p-2} (\G_2(u,u)+p(p-2)A_u^2)\\
    & \geq \G(u)^{\frac{p-2}{2}}\big(-\lambda(p-1)u^{(p-2)}\G(u) + \lambda(p-2)u^{(p-1)}A_u\big) + \frac{\lambda^2 u^{2p-2}}{n}\\
    &\quad+\frac{n}{n-1}\Big[\frac{\lambda u^{(p-1)}}{n}+(p-1)\G(u)^{\frac{p-2}{2}}A_u\Big]^2+\kappa\G(u)^{p-1}
\end{align*}

Hence in both cases, we have $\frac{1}{p}\L_p^u\big(\Gamma(u)^\frac{p}{2}-\phi(u)\big)\geq (\kappa-{\Tilde{\kappa}})\G(u)^{p-1}>0$, which is a contradiction with the second derivative test. Therefore we conclude that $Z_1\leq 0$ on $M$, which implies our gradient comparison result.
\end{proof}

\begin{remark}
When $1<p<2$ we know that $u\in C^{2,\alpha}$ near $x_0$, hence the Bochner formula can not be directly applied to $x_0$. In this case notice that $u$ does not vanish identically in a neighborhood of $x_0$, we can choose $x'\to x_0$ with $u'(x')\neq 0$. As we apply the Bochner formula at $x'$, The first term $\Gamma(u)^{\frac{p-2}{2}}\Gamma(L_pu,u)=-\lambda\Gamma(u)^{\frac{p-2}{2}}\Gamma(u^{(p-1)},u)$ since $u$ is a eigenfunction. Now this diverging term will cancel with $-\lambda(p-1)u^{p-2}\Gamma(u)^\frac{p}{2}$ in the expression of $\frac1p\L_p^u(\phi(u))$, which makes it possible to define $\frac{1}{p}\L_p^u\big(\Gamma(u)^\frac{p}{2}-\phi(u)\big)(x_0)$ to be the limit of $\frac1p\L_p^u(\phi(u))(x')$ as $x'\to x_0$. Therefore the previous proof still works when $1<p<2$.
\end{remark}
\section{One-dimensional Models}

In this section we will study the one-dimensional comparison model ODEs and discuss some fine properties of their solutions. Let $n\geq 2$, $p > 1$ be fixed. 
\subsection{The Model ODE}
Let $a\in\mathbb{R}$, $p > 1$ be fixed, we will consider the following form of initial value problem:
\begin{align}
\label{Model Equation}
\begin{cases}
\frac{d}{dt}(w')^{(p-1)}-T(t)\cdot(w')^{(p-1)}+\lambda w^{(p-1)}=0\\
w(a)=-1,\quad w'(a)=0
\end{cases}
\end{align}
where $T$ is defined over a subset of $\mathbb{R}$, to be specified according to the cases $\kappa > 0$ or $\kappa < 0$. To study this ODE we first define the $p$-$\sin$ and $p$-$\cos$ functions.

\begin{definition}
For every $p\in(1,\infty)$, let $\pi_p$ be defined by:
$$\pi_p=\int_{-1}^1\frac{ds}{(1-s^p)^\frac1p}=\frac{2\pi}{p\sin(\pi/p)}$$
The $C^1$ periodic function $\sin_p:\mathbb{R}\to[-1,1]$ is defined via the integral on $[-\frac{\pi_p}{2},\frac{3\pi_p}{2}]$ by
$$\begin{cases}
t=\int_{0}^{\sin_p(t)}(1-s^p)^{-\frac1p}ds\qquad &\text{ if }t\in\big[-\frac{\pi_p}{2},\frac{\pi_p}{2}\big]\\
\sin_p(t)=\sin_p(\pi_p-t)\qquad &\text{ if }t\in\big[\frac{\pi_p}{2},\frac{3\pi_p}{2}\big]
\end{cases}$$
and we extend it to a periodic function on $\mathbb{R}$. Let $\cos_p(t)=\frac{d}{dt}\sin_p(t)$, and we have the following identity which resembles the case of usual $\sin$ and $\cos$:
$$|\sin_p(t)|^p+|\cos_p(t)|^p=1.$$
\end{definition}

Let us use the Pr\"{u}fer transformation to study the model equation (6.1).

\begin{definition}[Pr\"{u}fer transformation] 
Let $\alpha=\big(\frac{\lambda}{p-1}\big)^\frac{1}{p}$, then for some solution $w$ of the ODE, we define functions $e$ and $\phi$ by 
$$\alpha w=e\sin_p(\phi)\quad w'=e\cos_p(\phi).$$
\end{definition}

Standard calculation shows that $\phi$ and $e$ satisfies the following first order systems:
\begin{align}
&\begin{cases}
\phi'=\alpha-\frac{T}{p-1}\cos_p^{p-1}(\phi)\sin_p(\phi)\\
\phi(a)=-\frac{\pi_p}{2}
\end{cases}
\\
&\begin{cases}
\frac{d}{dt}\log(e)=\frac{T}{p-1}\cos_p^p(\phi)\\
e(a)=\alpha
\end{cases}
\end{align}

\subsection{Choice of $T$ in the case $\kappa < 0$}

When $\kappa < 0$, we define $3$ functions $\tau_i$ on $I_i\subset \mathbb{R}$, $i=1,2,3$:

\begin{itemize}
    \item[(1)] $\tau_1(t)=\sinh(\sqrt{-\kappa}t)$
\item[(2)] $\tau_2(t)=\exp(\sqrt{-\kappa}t)$
\item[(3)]$\tau_3(t)=\cosh(\sqrt{-\kappa}t)$
\end{itemize}

and let $\mu_i=\tau_i^{n-1}$. Now we let $T_i=-\mu_i'/\mu_i$ and we get:
 
\begin{itemize}
    \item[(1)] $T_1(t)=-(n-1)\sqrt{-\kappa}\text{ cotanh}(\sqrt{-\kappa}t)$, defined on $I_1=(0,\infty)$;
    \item[(2)] $T_2(t)=-(n-1)\sqrt{-\kappa}$, defined on $I_2=\mathbb{R}$;
    \item[(3)] $T_3(t)=-(n-1)\sqrt{-\kappa}\tanh(\sqrt{-\kappa}t)$, defined on $I_3=\mathbb{R}$.
\end{itemize}

\subsection{Choice of $T$ in the case $\kappa > 0$}
When $\kappa > 0$, let $\tau_0(t) = \cos(\sqrt{\kappa}t)$ and $\mu_0 = \tau_0^{n-1}$. Then $T_0=\mu_0'/\mu_0$ is defined as
$$T_0(t) = (n-1)\sqrt{\kappa}\tan(\sqrt{\kappa}t)\quad\text{defined on } I_0 = (\frac{-\pi}{2\sqrt{\kappa}}, \frac{\pi}{2\sqrt{\kappa}})$$ 
\subsection{Fine analysis of the model equation (6.1)}

The central question we need to address here is the existence of solution to (6.1) whose range matches the range of $u$. Due to the normalization that $\min w = \min u = -1$, we will consider the maximum of $w$. For this purpose we introduce some notations. For $a\in\mathbb{R}$, let $w_{i,a}$ be the solution to the equation (6.1) with $T = T_i$, and $b(i,a)$ be the first critical point of $w_{i,a}$ after $a$. If $w'_{i,a}(t)>0$ for $t>a$, then we say $b(i,a)=\infty$. Also let $\delta_{i,a} = b(i,a)-a$ and $m(i,a) = w_{i,a}(b(i,a))$. We shall prove the following statement in the current and next section:

\begin{theorem}
\label{Existence of Exact Comparison}
Under the same setting as Theorem \ref{Main Theorem}, let $u$ be an eigenfunction of $L_p$ operator, normalized so that $\min u = -1$ and $\max u\leq 1$. Then we have the following existence results:
\begin{itemize}[leftmargin=0.3in]
    \item[(i)] ($\kappa > 0$) There is some $a\in I_0$ and a solution $w_{0,a}$ such that $m(0,a) = \max u$.
    \item[(ii)] ($\kappa < 0$) There is some $a\in\mathbb{R}$, $i\in\{1,2,3\}$ and a solution $w_{i,a}$ such that $m(i,a)=\max u$.
\end{itemize}
\end{theorem}

To prove Theorem \ref{Existence of Exact Comparison}, first we establish the following existence and uniqueness of a solution to the model equation.

\begin{proposition}
\label{Existence and Uniqueness}
There is a unique solution to the initial value problem (\ref{Model Equation}) with $T = T_i$, $i=1,2,3$ in the following cases:(1) $\kappa > 0$ and $a\in I_0 \cup \{-\pi/(2\sqrt{\kappa})\}$ and (2) $\kappa < 0$ and $a\in I_i \cup \{0\}$.
\end{proposition}

\begin{proof}[Proof of Proposition 6.1]
In the case $a\in I_i$ for $i=0,1,2,3$ we obtain the existence and uniqueness result from the fact that $T_i$ is a Lipschitz continuous function starting with $a$. Hence we need to confirm the boundary cases only. When (1) $\kappa>0$ and $a=-\pi/(2\sqrt{\kappa}t)$ and (2) $\kappa<0$ and $a = 0$ for model $T_1$, we can use fixed-point theorem argument to prove the existence and uniqueness of the solution by slightly modifying the proof in \cite{Wa98}, section 3.
\end{proof}

Then we shall look at the case $\kappa>0$. In order to find $a$ such that $w_{0,a}$ matches the maximum of $u$, we use the continuous dependence of $m(0,a)$ on $a$. We need to show that
\begin{proposition}
\label{Kappa Positive Odd Solution}
Fix $\alpha>0$, $n\geq 1$ and $\kappa>0$. Then there always exists a unique $\bar{a}\in I_0$ such that the solution $w_{3,-\bar{a}}$ is odd, and in particular, the maximum of $w$ restricted to $[-\bar{a},\bar{a}]$ is $1$.
\end{proposition}

The proof of Proposition \ref{Kappa Positive Odd Solution} requires certain weaker estimate of $\lambda$. We define the first Neumann eigenvalue of the equation (\ref{Model Equation}) on $I_0$ to be
$$\lambda_0 := \inf\bigg\{\frac{\int_{I_0}\cos^{n-1}(t)|w'|^p dt}{\int_{I_0}\cos^{n-1}(t)|w|^p dt},\quad w\in W^{1,p}(I_0)\setminus\{0\}\bigg\}$$
First we claim that 
\begin{lemma}
\label{smaller means odd positive}
If $\lambda\leq\lambda_0$, $\kappa>0$, then equation (\ref{Model Equation}) admits a odd solution $w$ such that $w'(t)>0$ for all $t\in I_0$.
\end{lemma}

\begin{proof}
Consider the initial value problem starting with $0$:
\begin{align}
\begin{cases}
\frac{d}{dt}(w')^{(p-1)}-T_0(t)\cdot(w')^{(p-1)}+\lambda w^{(p-1)}=0\\
w(0)=0, w'(0)>0
\end{cases}
\end{align}
This problem admits a solution $w$ up to $\pi/(2\sqrt{\kappa})$, the singularity of $T_0$, which can be extended to an odd solution on $I_0$. Now we claim that $w'(t)>0$ for all $t\in I_0$. Suppose for some $t_0\in (0, \pi/(2\sqrt{\kappa})$, $w'(t_0)=0$. This is an eigenfunction corresponding to $\lambda$, hence by the monotonicity of first eigenvalue, we get $\lambda > \lambda_0$, which contradicts with our assumption that $\lambda\leq \lambda_0$. Therefore, we have $w'(t)>0$ on $I_0$.
\end{proof}

From Lemma \ref{smaller means odd positive}, we can get a weaker bound on $\lambda$:

\begin{lemma}
When $\kappa>0$, we have $\lambda > \lambda_0$.
\end{lemma}
\begin{proof}
Suppose that on the contrary, $\lambda\leq \lambda_0$, then by Lemma \ref{smaller means odd positive}, we get an odd function $w$ such that $w'>0$ on $I_0$. Suppose the first eigenfuntion $u$ is scaled so that $-1 = \min u\leq \max u\leq 1$.
When $w$ is bounded, we can scale $w$ so that $\max w = w(\pi/(2\sqrt{\kappa}))=\max u$. Pick $x,y\in M$ such that $u(x) = -\max u$ and $u(y) = \max u$. Then by the gradient comparison theorem,
$$D \geq \text{dist}_M (x,y) \geq w^{-1}(u(y)) - w^{-1}(u(x)) = \frac{\pi}{\sqrt{\kappa}}$$
which is a contradiction. When $w$ is unbounded, we can choose $c_k\searrow 0$ and $t_k \nearrow \pi/(2\sqrt{\kappa})$ such that 
$c_k w(t_k) = \max u$. Again using the gradient comparison theorem, we can prove $D \geq 2t_k\to \pi/\sqrt{\kappa}$, which again gives a contradiction. Hence $\lambda > \lambda_0$.
\end{proof}

\begin{proof}[Proof of Proposition \ref{Kappa Positive Odd Solution}]
Using the Pr\"{u}fer Transformation, we consider the following initial value problem:
\begin{align}
\label{phi initial 0}
\begin{cases}
\phi'(t)=\alpha-\frac{T_0(t)}{p-1}\cos_p^{p-1}(\phi)\sin_p(\phi)\\
\phi(0)=0
\end{cases}
\end{align}
Since $\lambda > \lambda_0$, equation \ref{phi initial 0} has a solution $\phi$ such that $\phi(\hat{a}) = \pi_p/2$ for some $\hat{a}<\pi/(2\sqrt{\kappa})$. This implies that $e$ achieves maximum at $\hat{a}$ from the equation satisfied by $e$.  Therefore we conclude that $w = e\sin\phi/\alpha$ achieves maximum at $\hat{a}$, and $w$ can be extended to an odd function on $[-\hat{a},\hat{a}]$ such that $w'(-\hat{a}) = 0$, i.e. $w$ is a solution to the model equation (\ref{Model Equation}).
\end{proof}

By proposition \ref{Kappa Positive Odd Solution}, we know that $m(0,-\hat{a}) = 1$. Now we show the continuous monotonicity of $m(0,a)$, and first we need the following lemma to confirm continuity at the left endpoint:

\begin{lemma}
\label{Continuity of Maximum}
$$\lim_{a\to -\pi/(2\sqrt{\kappa})}m(0,a) = m(0,-\pi/(2\sqrt{\kappa})).$$
\end{lemma}

\begin{proof}[Proof of Lemma \ref{Continuity of Maximum}]
The idea of proof is from Proposition 1 of \cite{BQ00}. We will show that for any $T<\pi/(2\sqrt{\kappa})$ and $x\in(-\pi/(2\sqrt{\kappa}), T]$, we have
$$\lim_{a\to -\pi/(2\sqrt{\kappa})}w_a(x) = w_{-\pi/(2\sqrt{\kappa})}(x)\quad\text{and}\quad\lim_{a\to -\pi/(2\sqrt{\kappa})}w_a'(x) = w'_{-\pi/(2\sqrt{\kappa})}(x)$$
We denote $w_{-\pi/(2\sqrt{\kappa})}(x)$ by $w(x)$. We consider the function $W_a = w_a^{(p)} - w^{(p)}$, and know that the model equation \ref{Model Equation} can be written as
$$(\rho [w^{(p)}]')' + \lambda \rho w^{(p)} = 0$$
where $\rho(x) = \cos^{n-1}(\sqrt{\kappa}x)$. Hence we have
$$(\rho W_a')' + \lambda \rho W_a = 0\qquad\text{ on }[a,T)$$
Integrating the above equation over $[a,x]$ we get
$$\rho(x) W_a'(x)- \rho(a) W_a’(a) = -\lambda\int_a^x \rho(t) W_a(t)dt$$
Since $W_a'(a) = -(w^{(p)})'(a)$, we get
\begin{align}
\label{W prime}
W_a'(x) = -(w^{(p)})'(a)\frac{\rho(a)}{\rho(x)} -\lambda\int_a^x \frac{\rho(t)}{\rho(x)} W_a(t)dt
\end{align}
By another integration over $[a,x]$, we have
\begin{align}
\label{W}
    W_a(x) = W_a(a) -[w^{(p)}]'(a)\int_a^x\frac{\rho(a)}{\rho(y)}dy -\lambda\int_a^x W_a(t)\int_t^x \frac{\rho(t)}{\rho(y)}dydt
\end{align}
We know that as $T$ is fixed, 
$$\bigg|\int_t^x \frac{\rho(t)}{\rho(y)}dy\bigg|< C(n,T)$$
Since $W_a(a)\to 0$ and $[w^{(p)}]'(a)\to 0$ as $a\to -\pi/(2\sqrt{\kappa})$, the we have $W_a(x)\to 0$ and $W_a'(x) \to 0$ as $a\to-\pi/(2\sqrt{\kappa})$ by equation (\ref{W prime}) and (\ref{W}).
\end{proof}

\begin{proposition}
$m(0,a)$ is an continuous monotonic function of $a$ on $[-\pi/(2\sqrt{\kappa}), -\hat{a}]$.
\end{proposition}

\begin{proof}
First we show that $m(0,a)$ is an invertible function for $a\in [-\pi/(2\sqrt{\kappa}), -\hat{a}]$. Suppose there are $a$ and $a'$ such that $m(0,a) = m(0,a')$. Then since $w_{0,a}$ and $w_{0,a'}$ have same range and both are invertible functions on $[a,b(a)]$ and $[a',b(a')]$ respectively, by the gradient comparison theorem \ref{Gradient Comparison}, we have
$$w_{0,a}'\circ w_{0,a}^{-1} = w_{0,a'}'\circ w_{0,a'}^{-1}$$
and hence $w_{0,a}(x) = w_{0,a'}(x-a' + a)$, i.e. identical under a translation. However, by the $T_0$ model equation, this can only happen when $a=a'$. Therefore $m(0,a)$ is invertible, and it is monotonic.\\
To see the continuity of $m(0,a)$, note that when $a>-\pi/(2\sqrt{\kappa})$, the continuous dependence of the solution on the initial value problem is automatic. When $a = -\pi/(2\sqrt{\kappa})$, Lemma \ref{Continuity of Maximum} shows that $m(0,a)$ is continuous.
\end{proof}

Now let us turn to the case $\kappa<0$, which is more delicate. We will get a similar result as Proposition \ref{Kappa Positive Odd Solution}:

\begin{proposition}[\cite{NV14} Proposition 6.1]
\label{Kappa Negative Odd Solution}
Fix $\alpha>0$, $n\geq 1$ and $\kappa<0$. Then there always exists a unique $\bar{a}>0$ such that the solution $w_{3,-\bar{a}}$ is odd, and in particular, the maximum of $w$ restricted to $[-\bar{a},\bar{a}]$ is $1$.
\end{proposition}

By studying the equation of $\phi$ one can show that there is a critical value $\bar{\alpha}$ at which the oscillatory behavior of $w$ changes. For the modeal $T_3$, we have
\begin{lemma}[\cite{NV14} Proposition 6.4]
There exists a limiting value $\bar{\alpha}>0$ such that for $\alpha>\bar{\alpha}$ we have $\delta(3,a)<\infty$ for every $a\in\mathbb{R}$. For $\alpha<\bar{\alpha}$, we have $$\lim_{t\to\infty}\phi_{3,a}(t)<\infty\qquad\text{ for all }a\in\mathbb{R}.$$
for sufficiently large $a$ we have 
$$-\frac{\pi_p}{2}<\lim_{t\to\infty}\phi_{3,a}(t)<0\qquad\text{ and }\delta(3,a)=\infty.$$
When $\alpha=\bar{\alpha}$, we have $\lim_{a\to\infty}\delta(3,a)=\infty.$
\end{lemma}

For model $T_1$ we get the following result:

\begin{lemma}[\cite{NV14} Proposition 6.5]
There exist $\bar{\alpha}>0$ such that when $\alpha>\bar{\alpha}$ then $\delta(1,a)<\infty$ for all $a\in[0,\infty)$. If $\alpha\leq \bar{\alpha}$ then  $\phi_{1,a}$ has finite limit at infinity and $\delta(1,a)<\infty$ for all $a\in[0,\infty)$.
\end{lemma}

Both cases $\alpha<\bar{\alpha}$ and $\alpha\geq\bar{\alpha}$ need to be considered in proving the case (2) of Theorem \ref{Existence of Exact Comparison}. When $\alpha<\bar{\alpha}$ we can always use model $T_3$ to produce the whole range comparison solutions $w$, i.e. $0< \max w \leq 1$, and when $\alpha\geq\bar{\alpha}$ we have restriction on the maximum value that $u$ can achieve. More precisely we have:

\begin{lemma}[\cite{NV14} Proposition 6.6]
Let $\alpha\leq\bar{\alpha}$. Then for each $0<\max u\leq 1$, there is an $a\in[-\bar{a},\infty)$ such that $m(3,a)=\max u$.
\end{lemma}

We can also see that model $T_2$ is translation invariant, hence for all $a\in[0,\infty)$, $m(2,a)=m_2$ is a constant. For model $T_1$ and $T_3$ we have

\begin{lemma}[\cite{NV14} Proposition 6.7]
If $\alpha>\bar{\alpha}$, then $m(3,a)$ is a decreasing function of $a$, while $m(1,a)$ is an increasing function of $a$ and 
$$\lim_{a\to\infty}m(3,a)=\lim_{a\to\infty}m(1,a)=m_2.$$
\end{lemma}

Combining the Proposition \ref{Kappa Negative Odd Solution} and Lemmas above, we know that in the case $\kappa<0$, if $m(1,0)\leq \max u\leq 1$, there is always a model solution $w$ to $T_1,T_2$ or $T_3$ such that $\max w = \max u$.

\subsection{Diameter Comparison}

In order to get the eigenvalue comparison with one-dimensional moder of the same diameter bound, we still need to understand how $\lambda_D$ varies with the diameter. Again we will follow \cite{NV14}.

\begin{definition}
We define the minimum diameter of the one-dimesional model associated with $\lambda$ to be
$$\bar{\delta}_i(\lambda)=\min\{\delta(i,a)|i=0,1,2,3,a\in I_i\}$$
\end{definition}

The following propositions deals with the lower bound of $\bar{\delta}_i(\lambda)$ for $i=0,1,2,3$:

\begin{proposition}
\label{delta lower bound positive}
For $i=0$ and any $a\in I_0$, we have $\delta(0,a)>2\hat{a}$, where $\hat{a}<\pi/(2\sqrt{\kappa})$ is such that $w_{0,-\hat{a}}$ is odd. 
\end{proposition}

\begin{proof}
The proof here is based on the symmetry and convexity of the model $T_0$. See \cite{NV14} Proposition 8.4 for the proof.
\end{proof}

For the case $\kappa < 0$, we cite the following results from \cite{NV14}:

\begin{proposition}[\cite{NV14}, Proposition 8.2]
For $i=1,2$ and any $a\in I_i$, we have $\delta(i,a)>\frac{\pi_p}{\alpha}$. 
\end{proposition}

Model 3 needs a little bit careful attention. For this one we notice first that there is always $\bar{a}>0$ with an odd solution for initial data at $-\bar{a}$. Namely $w_{3,\bar{a}}$ is odd function with min $-1$ and max $1$. This is a critical situation which minimizes the diameter $D$ given $\lambda$:

\begin{proposition}[\cite{NV14}, Proposition 8.4]
For $i=3$ and $a\in\mathbb{R}$, we have
$$\delta(3,a)\geq\delta(3,\bar{a})=2\bar{a}$$
and if $a\neq -\bar{a}$, the inequality is strict.
\end{proposition}

It is also easy to see from the ODE for $\phi$ when $i=3$ that, $\phi'>\alpha$. Therefore $\delta(3,-\bar{a})<\frac{\pi_p}{\alpha}$. Also from this we have $\delta(3,-\bar{a})$ is strictly decreasing function of $\alpha$, so as to $\lambda_D$. This means that $\bar{\delta}(\lambda)$ is a strictly decreasing function. Thus if we see $\lambda$ as a function of $\delta$, we also have the monotonicity: if $\delta_1\leq\delta_2$, we have
$$\lambda(\delta_1)\geq\lambda(\delta_2).$$

\section{Maximum of Eigenfunctions}

In this section we are going to compare the maximum of the eigenfunctions $u$ and the model functions $w$. First, we define a measure on the interval $[a,b(a)]$ which is essentially the pullback of the volume measure on $M$ by $w^{-1}\circ u$.  By the ODE satisfied by $w$, $w''$ is positive before $w$ hits its first zero.

First we have a theorem which can be seen as a comparison between the model function and the eigenfunction.

\begin{theorem}(Theorem 34 \cite{NV14})
Let $u$ and $w$ be as above and define
$$E(s):=-\exp{\bigg(\int_{t_0}^s\frac{w^{(p-1)}}{w'^{(p-1)}}dt\bigg)}\int_a^sw^{(p-1)}d\mu$$
then $E$ is increasing on $(a,t_0]$ and decreasing on $[t_0,b)$.

\end{theorem}

This result is equivalent to the following statement:

\begin{theorem}(Theorem 35,\cite{NV14})
Under the hypothesis of Theorem 6.1 the function
$$E(s):=\frac{\int_a^sw^{(p-1)}d\mu}{\int_a^sw^{(p-1)}t^{n-1}dt}=\frac{\int_{u\leq w(s)}u^{(p-1)}dm}{\int_a^sw^{(p-1)}t^{n-1}dt}$$
is increasing on $(a,t_0]$ and decreasing on $[t_0,b)$.
\end{theorem}

To prove the maximum comparison we study the volume of a small ball around the minimum of $u$. By the gradient comparison we have the following:

\begin{lemma}
For $\epsilon$ sufficiently small, the set $u^{-1}[-1,-1+\epsilon)$ contains a ball of radius $w^{-1}(-1+\epsilon)-a$. 
\end{lemma}

Now we can prove the maximum comparison, by combining Bishop-Gromov and the following:

\begin{theorem}
Let $n\geq N$ and $n>1$. If $u$ is an eigenfunction satisfying $\min u=-1=u(x_0)$ and $\max u\leq m(1,0)=w_{1,0}(b(1,0))$, then there exists a constant $c>0$ such that for all $r$ sufficiently small, we have 
$$m(B_{x_0}(r))\leq cr^n.$$
\end{theorem}
\begin{proof}
To keep notations short, let $w=w_{1,0}$. Let $\epsilon$ be small such that $-1+\epsilon<-2^{-p+1}$. Then we have $u^{(p-1)}<-\frac{1}{2}$ when $u<-1+\epsilon$. Let $t_0$ be the first zero of $w$, then by Theorem 6.1 we have $E(t)\leq E(t_0)$. Therefore by Theorem 6.2 we get
$$m(B_{x_0}(r_\epsilon))\leq C\int_{u\leq -1+\epsilon}u^{(p-1)}dm\leq CE(t_0)\int_a^{w^{-1}(-1+\epsilon)}w^{(p-1)}t^{n-1}dt\leq C'r_\epsilon^n$$
Since $\epsilon$ can be arbitrarily small, we have the claim holds for $r$ sufficiently small.
\end{proof}

\begin{corollary}
Let $n\geq N>1$, and $u$ is an eigenfunction of $L_p$ with $\min u = -1$ We have the following maximum comparison result:
\begin{itemize}[leftmargin=0.3in]
    \item[(1)] If $\kappa>0$, and $w_{0,-\pi/(2\sqrt{\kappa})}$ is the corresponding eigenfunction. Then $\max u \geq m(0,-\pi/$ $(2\sqrt{\kappa}))$.
    \item[(2)] If $\kappa<0$, and $w_{1,\hat{a}}$ is the corresponding eigenfunction. Then $\max u \geq m(1,0)$.
\end{itemize}  
\end{corollary}

\begin{proof}
Let $m$ denotes $m(0,-\pi/(2\sqrt{\kappa}))$ or $m(1,0)$ in either cases, and suppose that $\max u<m$. Since $m$ is the least possible value among $\max w$ for all model solutions $w$, by continuous dependence of the solution of model equation on $n$, we can find $n'>n$ so that $\max u$ is still less that the maximum of the correspoding model equation. Since $BE(\kappa,n')$ is still satisfied, we have by Theorem 6.3, that $m(B_{x_0}(r))\leq cr^{n'}$ for $r$ sufficiently small. However by Bishop-Gromov volume comparison we have $m(B_{x_0}(r))\geq Cr^N$. This is a contradiction since $n'>n\geq N$.
\end{proof}
\section{Proof of Theorem 1.1}
Now we can combine the gradient and maximum comparison, together with properties of the model equation to show the Theorem \ref{Main Theorem}.

\begin{theorem}
Let $M$ be compact and connected and $L$ be an elliptic diffusion operator with invariant measure $m$. Assume that $L$ satisfies $BE(\kappa,N)$ where $\kappa\neq 0$. Let $D$ be diameter defined by the intrinsic distance metric on $M$. Let $u$ be an eigenfunction associated with $\lambda$
satisfying Neumann boundary condition if $\p M \neq \emptyset$, where $\lambda$ is the first nonzero eigenvalue of $L_p$. Then denoting $w^{(p-1)} := |w|^{p-2}w$, we have
\begin{itemize}[leftmargin=0.5in]
    \item[(1)] When $\kappa>0$, assuming further that $D\leq \pi/\sqrt{\kappa}$, we have a sharp comparison:
    $$\lambda\geq \lambda_D$$
    where $\lambda_D$ is the first nonzero eigenvalue of the following Neumann eigenvalue problem on $[-D/2,D/2]$:
    $$\frac{d}{dt}\big[(w')^{(p-1)}\big]-(n-1)\sqrt{\kappa}\tan(\sqrt{\kappa} t)(w')^{(p-1)}+\lambda w^{(p-1)}=0$$
    \item[(2)] When $\kappa<0$, we have a sharp comparison:
    $$\lambda\geq \lambda_D$$
    where $\lambda_D$ is the first nonzero eigenvalue of the following Neumann eigenvalue problem on $[-D/2,D/2]$:
    $$\frac{d}{dt}\big[(w')^{(p-1)}\big]+(n-1)\sqrt{-\kappa}\tanh(\sqrt{-\kappa t})(w')^{(p-1)}+\lambda w^{(p-1)}=0$$
\end{itemize}
\end{theorem}

\begin{proof}
We scale $u$ so that $\min u=-1$ and $\max u\leq 1$. By Proposition \ref{Existence of Exact Comparison} we can find a model function $w_{i,a}$ such that $\max u=\max w_{i,a}$. By the gradient comparison theorem, $\Gamma(w_{i,a}^{-1}\circ u)\leq 1$.  Let $x$ and $y$ on $M$ be points where $u$ attains maximum and minimum, then we have
$$D\geq |w_{i,a}^{-1}\circ u(x)-w_{i,a}^{-1}\circ u(y)|=w_{i,a}^{-1}(m(i,a))-w_{i,a}^{-1}(-1)=\delta(i,a,\lambda)\geq \delta(i,\bar{a})$$
Therefore by the monotonicity of eigenvalue of the model equation, we have that 
$$\lambda\geq \lambda_D.$$
To check the sharpness of this result when $\kappa < 0$, we have the following examples: let $$M_i=[-D/2,D/2]\times_{i^{-1}\tau_3}S^{n-1}$$
be a warped product where $S^{n-1}$ is the standard unit sphere, and $\tau_3(t)=\cosh(\sqrt{-\kappa}t)$. If we consider $L$ being the classical Laplacian on $M$, then standard computation shows that $M_i$ has $\text{Ric}\geq -(n-1)\kappa$ and geodeiscially convex boundary. Hence it also satisfy the $BE(\kappa,N)$ condition. If we take $u(t,x)=w(t)$ where $w$ is the solution to our one-dimensional model equation with $\lambda=\lambda_D$. Since the diameter of $M_i$ tends to $d$ as $i\to\infty$, we see that the first eigenvalue on $M_i$ converges to $\lambda_d$, which shows the sharpness of our lower bound when $\kappa < 0$. For $\kappa > 0$, the round sphere $S^{n-1}(\frac{\pi}{\sqrt{\kappa}})$ serves as a model for sharp lower bound of $\lambda$.
\end{proof}


\newpage

\bibliographystyle{plain}
\bibliography{ref.bib}

\end{document}